\documentclass[twoside,a4paper,11pt]{amsart}
\pdfoutput=1

\usepackage{amsmath}
\usepackage{amsthm}
\usepackage{amssymb}
\usepackage{amsfonts}
\usepackage{amsxtra}

\usepackage{enumitem}
\setenumerate{label={\rm (\alph{*})}}

\usepackage[frak]{paper_diening}

\newtheorem{theorem}{Theorem}[section]
\newtheorem{lemma}[theorem]{Lemma}
\newtheorem{corollary}[theorem]{Corollary}

\newtheorem{remark}[theorem]{Remark}

\newtheorem{definition}[theorem]{Definition}

\newtheorem{assumption}[theorem]{Assumption}

\numberwithin{equation}{section}

\newcommand{\meantmp}[2]{#1\langle{#2}#1\rangle}
\newcommand{\mean}[1]{\meantmp{}{#1}}

\newcommand{\Rn}{{\setR^n}}
\newcommand{\RN}{{\setR^N}}
\newcommand{\RNn}{\setR^{N\times n}}

\newcommand{\px}{{p(\cdot)}}

\newcommand{\PP}{\mathcal{P}}
\newcommand{\PPln}{\mathcal{P}^{{\log}}}

\providecommand{\bfAT}{{\bfA_{\mathcal{T}}}}

\begin{document}

\title{Finite element approximation of the $\px$-Laplacian}
\author{Dominic Breit, Lars Diening and  Sebastian Schwarzacher}
\address{LMU Munich, Institute of Mathematics, Theresienstr. 39,
  80333-Munich, Germany}

\begin{abstract}
  We study a~priori estimates for the $\px$-Laplace Dirichlet problem, $-\mathrm{div}(\abs{\nabla \bfv}^{p(\cdot)-2} \nabla
  \bfv) = \bff$. We show that the gradients of the finite element
  approximation with zero boundary data converges with rate
  $O(h^\alpha)$ if the exponent~$p$ is $\alpha$-H{\"o}lder continuous. The
  error of the gradients is measured in the so-called quasi-norm,
  i.e. we measure the $L^2$-error of $\abs{\nabla
    \bfv}^{\frac{p-2}{2}} \nabla \bfv$.
\end{abstract}

\maketitle

\date{\small \today}


\section{Introduction}
\label{sec:introduction}
\noindent
In recent years there has been an extensive interest in the field of
variable exponent spaces $L^\px$. Different from the classical
Lebesgue spaces $L^p$, the exponent is not a constant but a function
$p=p(x)$. We refer to the recent books \cite{CruFio13,DieHHR11} for a
detailed study of the variable exponent spaces, although the
definition of the spaces goes already back to Orlicz~\cite{Orl31}.

The increasing interest was motivated by the model for
electrorheological fluids~\cite{RajR96,Ruz00}. Those are smart
materials whose viscosity depends on the applied electric field.
Nowadays it is possible to change the viscosity locally by a factor
of~1000 in~1ms. This is modeled via a dependence of the viscosity on a
variable exponent. Electrorheological fluids can for example be used
in the construction of clutches and shock absorbers.

Further applications of the variable exponent spaces can be found in
the area of image reconstruction. Here, the change of the exponent is
used to model different smoothing properties according the edge
detector. This can be seen as a hybrid model of standard diffusion and
the TV-model introduced by~\cite{CheLevRao06}, see
also~\cite{BolChaEseSchVix09} and~\cite{HarHasLeNuo10} for an
overview over the related topics.

A model problem for image reconstruction as well as a starting point
for the study of electro-rheological fluids is the $\px$-Laplacian
system
\begin{align}\label{eq:1}
  -\mathrm{div}((\kappa+\abs{\nabla\bfv})^{\px-2}\nabla\bfv) =
  \bff\quad\text{in}\quad\Omega
\end{align}
on a bounded domain $\Omega$, where $\kappa\geq0$. Here
$\bff:\Omega\rightarrow\RN$ is a given function and
$\bfv:\Omega\rightarrow\RN$ is the vector field we are seeking for.
The natural function space for solutions is the generalized Sobolev
space $(W^{1,\px}(\Omega))^N$ which is the set of all functions
$\bfu\in (L^{\px}(\Omega))^N$ with distributional gradient $\nabla \bfu\in
(L^{\px}(\Omega))^{N\times n}$. Here we have
\begin{align}\label{eq:Lpx}
  L^{\px}(\Omega)=\Bigset{f\in
    L^1(\Omega):\,\,\int_{\Omega}\abs{f}^{\px}\,dx < \infty}.
\end{align}
A major breakthrough in the theory of variable exponent spaces was the
fact that the right condition on the exponent was found: the {\em
  log-H{\"o}lder continuity} (see Section~\ref{sec:prel}).  In fact for
such exponents $\px$ bounded away from $1$ the Hardy-Littlewood
maximal operator is bounded in $L^\px(\Omega)$,
see~\cite{CruFN03,Die04} and \cite[Theorem~4.3.8]{DieHHR11}. This is a
consequence of the so called {\em key estimate for variable exponent
  spaces}, which roughly reads
\begin{align}
  \label{eq:keypre}
  \Bigg(\dashint_Q |f|\, dx\Bigg)^{p(y)} \le c\,\dashint_Q
  |f|^{p(x)}\, dx + \text{error}
\end{align}
where $Q$ is a ball or cube, $y\in Q$ and the ``error'' denotes an
appropriate error term, which is essentially independent of~$f$;  see
Theorem~\ref{thm:jensenpx} for the precise statement. Once the maximal
operator is bounded many techniques known for Lebesgue spaces can be
extended to the setting of variable exponents.

Under the assumptions on $\px$ mentioned above the existence of a weak
solution to \eqref{eq:1} is quite standard provided $\bff\in
(L^{\px'}(\Omega))^N$, i.e. there is a unique function
$\bfv\in(W^{1,\px}_0(\Omega))^N$ such that
\begin{align}
  \label{eq:app_sys}
  \int_\Omega \bfA (\cdot,\nabla \bfv) \cdot\nabla\bfpsi\, dx =
  \int_\Omega \bff \cdot \bfpsi\, dx,
\end{align}
for all $\bfpsi \in (W^{1,\px}_{0}(\Omega))^N $ holds. There is a huge
bulk of literature regarding properties of solutions, see for
instance~\cite{AceM01,AceM05,DiePhd} and it seems that the problem is well
understood from the analytical point of view. 

However, there is only few literature about numerical algorithms for
discrete solutions of~\eqref{eq:app_sys}. In \cite{DelLomMar12} the
convergence of discontinuous Galerkin FEM approximations is studied.
They prove strong convergence of the gradients but without any explicit rate.
The paper~\cite{CarHaePro10} is concerned with the
corresponding generalized Navier-Stokes equations which are certainly
much more sophisticated then~\eqref{eq:1}. The convergence of the
finite element approximation is shown without convergence rate. To our
knowledge no quantitative estimate on the convergence rate for
problems with variable exponents is known. To derive such an estimate
is the main purpose of the present paper.

Let us give a brief overview about conformal finite element
approximation for more classical equations.  Let $\Omega$ be a
polyhedron decomposed in simplices with lengths bounded by $h$.  In
case of the Laplace equation $-\Delta\bfv=\bff$ one can easily show
that
\begin{align*}
 \norm{\nabla\bfv-\nabla\bfv_h}_{2}\leq \,c\, h\norm{\nabla^2\bfv}_2,
\end{align*}
where $\bfv_h$ is the finite element solution. This is based on best
approximation and interpolation results and rather classical.

Let us turn to the $p$-Laplacian, i.e. \eqref{eq:1} with constant $p$.
As was firstly observed in \cite{BarL94} the convergence of nonlinear
problems should be measured with the so-called quasi-norm. Indeed, in
\cite{EbmLiu05,DieR07} it was shown that
\begin{align}
\label{eq:perror}
 \norm{\bfF(\nabla\bfv)-\bfF(\nabla\bfv_h)}_{2}\leq \,c\, h\norm{\nabla\bfF(\nabla\bfv)}_2,
\end{align}
where $\bfF(\bfxi)=(\kappa+\abs{\bfxi})^{p-2}\bfxi$.  The result is
based on best approximation theorems and interpolations results in
Orlicz spaces. It holds also for more general nonlinear
equations~\cite{DieR07}.  Please remark, that $\bfF\in W^{1,2}$ is
classical see~\cite{Uhl77,LiuBar93rem} for smooth domains. In the case
of polyhedral domains with re-entrant corners certainly less regularity
is expected (as in the linear case),
see~\cite{Ebm01poly,EbmFre01,EbmLiuSte05}.  Let us mention, that for the
p-Laplacian results about rate optimality of the adaptive finite
element method has been shown in~\cite{DieK08,BelDieKre11} followed.

In this work finite element solutions to the $\px$-Laplacian under the
assumption that $p\in C^{0,\alpha}(\overline{\Omega})$ are studied,
see~\eqref{eq:weak_h}. To be precise we show in Theorem~\ref{thm:error}
that
 \begin{align}\label{eq:conv}
 \norm{\bfF(\cdot,\nabla\bfv)-\bfF(\cdot,\nabla\bfv_h)}_{2}\leq h^\alpha c(\norm{\bfF(\cdot,\nabla\bfv)}_{1,2}),
\end{align}
where $\bfF(\cdot,\bfxi)=(\kappa+\abs{\bfxi})^{\px-2}\bfxi$. If $p$ is
Lipschitz continuous we can recover the result for the constant
$p$-Laplacian~\eqref{eq:perror}. As usual the proof of~\eqref{eq:conv}
is mainly based on interpolation operators and a lemma of \Cea{} type.
Therefore we study in Section~\ref{sec:interpolation-wk-phi}
interpolation operators on Lebesgue spaces of variable exponents. In
order to do so we need an extension of the key estimate
\eqref{eq:keypre} which can be found in
Theorem~\ref{thm:jensenpxshift}. These purely analytic results are
improved tools in the variable Lebesgue and Sobolev theory and are
therefore of independent interest.

Finally we introduce a numerical scheme, where the exponent $\px$ is
locally approximated by constant functions, see~\eqref{eq:weak_h'}.
This seems more appropriate for practical use.  Naturally another
error term due to the approximation of the exponent occurs.  However,
we show, that this error does not harm the convergence rate
and~\eqref{eq:conv} still holds, see Theorem~\ref{thm:frozen}.

%
%
%
%
%
%
%
%
%

\section{Variable exponent spaces}
\label{sec:prel}
\noindent
For a measurable set $E \subset \Rn$ let $\abs{E}$ be the Lebesgue
measure of~$E$ and $\chi_E$ its characteristic function. For $0 <
\abs{E} < \infty$ and $f \in L^1(E)$ we define the mean value of $f$
over $E$ by
\begin{align*}
   \mean{f}_E:=\dashint_{E}f \,dx:=\frac{1}{\abs{E}}\int_E f \,dx.
\end{align*}
For an open set $\Omega \subset \Rn$ let $L^0(\Omega)$ denote the set
of measurable functions.

Let us introduce the spaces of variable exponents~$L^\px$.  We use the
notation of the recent book~\cite{DieHHR11}.  We define $\PP$ to
consist of all $p \in L^0(\Rn)$ with $p\,:\, \Rn \to [1,\infty]$
(called variable exponents).  For $p \in\PP$ we define $p^-_\Omega :=
\essinf_\Omega p$ and $p^+_\Omega := \esssup_\Omega p$.  Moreover, let
$p^+ := p^+_{\Rn}$ and $p^- := p^-_{\Rn}$.

For $p \in \PP$ the generalized Lebesgue space $L^\px(\Omega)$ is
defined as
\begin{align*}
  L^\px(\Omega):=\bigset{f \in L^0(\Omega) \,:\,
    \norm{f}_{L^\px(\Omega)} < \infty},
\end{align*}
where
\begin{align*}
  \norm{f}_{\px}:=\norm{f}_{L^\px(\Omega)}:=
  \inf\Biggset{\lambda>0\,:\,\int_{\Rn} \biggabs{
      \frac{f(x)}{\lambda}}^{p(x)} \,dx \leq 1}.
\end{align*}


We say that a function $\alpha\colon \Rn \to \setR$ is {\em
  $\log$-H{\"o}lder continuous} on $\Omega$ if there exists a constant $c
\geq 0$ and $\alpha_\infty
\in \setR$ such that
\begin{align*}
  \abs{\alpha(x)-\alpha(y)} &\leq \frac{c}{\log
    (e+1/\abs{x-y})} &&\text{and}& \abs{\alpha(x) - \alpha_\infty}
  &\leq \frac{c}{\log(e + \abs{x})}
\end{align*}
for all $x,y\in \Rn$. The first condition describes the so called
local $\log$-H{\"o}lder continuity and the second the decay condition.
The smallest such constant~$c$ is the $\log$-H{\"o}lder constant
of~$\alpha$. We define $\PPln$ to consist of those exponents $p\in
\PP$ for which $\frac{1}{p} \,:\, \Rn \to [0,1]$ is $\log$-H{\"o}lder
continuous. By $p_\infty$ we denote the limit of~$p$ at infinity,
which exists for $p \in \PPln$.  If $p \in \PP$ is bounded, then $p
\in \PPln$ is equivalent to the $\log$-H{\"o}lder continuity of $p$.
However, working with $\frac 1p$ gives better control of the constants
especially in the context of averages and maximal functions.
Therefore, we define $c_{\log}(p)$ as the $\log$-H{\"o}lder constant
of~$1/p$. Expressed in~$p$ we have for all $x,y \in \Rn$
\begin{align*}
  \abs{p(x) - p(y)} \leq \frac{(p^+)^2 c_{\log}(p)}{\log
    (e+1/\abs{x-y})} &&\text{and }& \abs{p(x) - p_\infty} &\leq
  \frac{(p^+)^2 c_{\log}(p)}{\log(e + \abs{x})}.
\end{align*}
\begin{lemma}
  \label{lem:pxpy} 
Let $p \in \PPln(\Rn)$ with $p^+<\infty$ and $m>0$. Then for every $Q \subset \Rn$ with $\ell(Q) \leq 1$, $\kappa\in [0,1]$ and $t\geq 0$, that holds $\abs{Q}^m \leq t \leq \abs{Q}^{-m}$
  \begin{align*}
(\kappa +
    t)^{p(x)-p(y)} \leq c
  \end{align*}
for all $x,y \in Q$. The constants depend on $c_{\log}(p),m$ and $p^+$.
\end{lemma}
\begin{proof}
  Recall that $\ell(Q) \leq 1$ and $\abs{Q} \leq 1$.  Since $\kappa \leq
  \abs{Q}^{-m}$ for $\kappa \in [0,1]$ we have $\abs{Q}^m \leq
  \kappa + t \leq 2\,\abs{Q}^{-m}$.  Thus the local $\log$-H{\"o}lder
  continuity of~$p$ implies
  \begin{align*}
    (\kappa + t)^{p(x)-p(y)} &\leq 2^{p^+}
    \big(\abs{Q}^{-\abs{p(x)-p(y)}}\big)^m \leq c,
  \end{align*}
  where the constant depends on $c_{\log}(p)$, $p^+$ and~$m$. 
\end{proof}

\subsection{Key estimate of variable exponents}
\label{ssec:keyestimate}

For every convex function~$\psi$ and every cube~$Q$ we have by
Jensen's inequality
\begin{align}
  \label{eq:jensen}
  \psi\bigg(\dashint_Q \abs{f(y)}\,dy\bigg) &\leq \dashint_Q \psi(
  \abs{f(y)})\,dy.
\end{align}
This simple but crucial estimate allows for example to transfer the
$L^1$-$L^\infty$ estimates for the interpolation operators to the
setting of Orlicz spaces, see~\cite{DieR07}. Therefore, it is necessary
for us to find a suitable substitute for Jensen's inequality in the
context of variable exponents.

Our goal is to control $ ( \dashint_Q \abs{f(y)} \,dy
)^{p(x)}$ in terms of $\int_Q \abs{f(x)}^{p(x)} \,dx$. If $p$ is
constant, this is exactly Jensen's inequality and it holds for all $f
\in L^p(Q)$. However, for~$p$ variable it is
impossible to derive such estimates for all $f$ and we have to
restrict ourselves to a certain set of admissible functions~$f$.
Moreover, an additional error term appears. The following statement is
a special case of~\cite[Corollary~1]{DieSch13key}, which is an
improvement of~\cite[Corollary~4.2.5]{DieHHR11} and related estimates
from~\cite{DieHHMS09,Schw10}.
\begin{theorem}[key estimate]
  \label{thm:jensenpx}
  Let $p \in \PPln(\Rn)$ with $p^+<\infty$. Then for every $m>0$ there
  exists $c_1>0$ only depending on $m$ and $c_{\log}(p)$ and $p^+$
  such that
  \begin{align}
    \label{eq:jensenpx}
    \bigg(\dashint_Q \abs{f(y)}\,dy \bigg)^{p(x)} &\leq c_1 \dashint_Q
    \abs{f(y)}^{p(y)} \,dy + c_1 \abs{Q}^m.
  \end{align}
  for every cube (or ball) $Q \subset \Rn$ with $\ell(Q) \leq 1$, all
  $x \in Q$ and all $f\in L^1(Q)$ with
  \begin{align*}
    \dashint_Q \abs{f}\,dy \leq \max \set{1, \abs{Q}^{-m}}.
  \end{align*}
\end{theorem}
Let us introduce the notation
\begin{align*}
  \phi(x,t) &:= t^{p(x)},
  \\
  (M_Q \phi)(t) &:= \dashint_Q \phi(x,t)\,dx,
  \\
  M_Q f &:= \dashint_Q \abs{f(x)}\,dx,
  \\
  \phi(f) &:= \phi( \cdot, \abs{f(\cdot)}).
\end{align*}
Then we can rewrite~\eqref{eq:jensenpx} as
\begin{align}
  \label{eq:jensenpx2}
  \phi(M_Q f)&\leq c\, M_Q(\phi(f)) + c\, \abs{Q}^{m}
\end{align}
for all $f$ with $M_Q f \leq \max \set{1, \abs{Q}^{-m}}$.

For our finite element analysis we need this estimate extended to the
case of shifted Orlicz functions. For constant~$p$ this has been done
in~\cite{DieR07}. We define the shifted functions $\phi_a$ for $a\geq
0$ by
\begin{align*}
  \phi_a(x,t) &:= \int_0^t \frac{\phi'(x,a+\tau)}{a+\tau} \tau\,d\tau.
\end{align*}
Then $\phi_a(x,\cdot)$ is the shifted N-function of $t \mapsto
t^{p(x)}$, see~\eqref{eq:def_shift}. Note that characteristics and
$\Delta_2$-constants of $\phi_a(x,\cdot)$ are uniformly bounded with
respect to~$a\geq 0$, see Section~\ref{sec:Orlicz spaces}.

\begin{remark}
  \label{rem:shift_in_out}
  It is easy to see that for all $a,t \geq 0$ we have
  \begin{align*}
    (M_Q \phi_a)(t) = (M_Q \phi)_a(t).
  \end{align*}
  If $p^+< \infty$, then $\Delta_2(\set{M_Q \phi_a}_{a \geq 0}) <
  \infty$.
\end{remark}

In the rest of this section we will prove the following theorem.
\begin{theorem}[shifted key estimate]
  \label{thm:jensenpxshift}
  Let $p \in \PPln(\Rn)$ with $p^+<\infty$. Then for every $m>0$ there
  exists $c>0$ only depending on $m$, $c_{\log}(p)$ and $p^+$
  such that
  \begin{align*}
    \phi_a(x,M_Q f)&\leq c\, M_Q(\phi_a(f)) + c\, \abs{Q}^{m}
  \end{align*}
  for every cube (or ball) $Q \subset \Rn$ with $\ell(Q) \leq 1$, all
  $x \in Q$ and
  all $f\in L^1(Q)$ with
  \begin{align*}
    a+\dashint_Q \abs{f}\,dy \leq \max \set{1, \abs{Q}^{-m}} =
    \abs{Q}^{-m}. 
  \end{align*}
\end{theorem}
\begin{proof}
  We split the proof into three case:
  \begin{enumerate}
  \item $M_Q f \geq a$,
  \item $M_Q f \leq a \leq \abs{Q}^m$,
  \item $M_Q f \leq a$ and $\abs{Q}^m \leq a$.
  \end{enumerate}
  {\bf Case: $M_Q f \geq a$}

  Let $f_1 := 2f \chi_{\set{\abs{f} \geq \frac a2}}$. Then $M_Q f \geq
  a$, implies
  \begin{align*}
    M_Q f_1 = M_Q (2f \chi_{\set{\abs{f} \geq \frac a2}}) &\geq M_Q f -
    \frac a 2 \geq \frac 12 M_Q f \geq \frac a2.
  \end{align*} 
  Thus with the $\Delta_2$-condition, Remark~\ref{rem:shift_in_out},
  Lemma~\ref{lem:shift_sim} (using $M_Q f_1 \gtrsim a$) we get
  \begin{align*} 
    \phi_a(x,M_Q f) &\leq c\, \phi_a(x,M_Q f_1)
    \sim c\, \phi(x,M_Q f_1).
    \\
    \intertext{Thus, the key estimate of Theorem~\ref{thm:jensenpx}
      implies}%
    \phi_a(x,M_Q f) &\leq c\, M_Q \big(\phi(f_1)\big) + c\, \abs{Q}^{m}.
    \\
    \intertext{Now by Lemma~\ref{lem:shift_sim} and the definition
      of~$f_1$ we have $\phi(f_1) \sim \phi_a(f_1)$ and $\abs{f_1}
      \leq 2\abs{f}$. Thus}%
    \phi_a(x,M_Q f)  &\leq c\, M_Q \big(\phi_a(f_1)\big) + c\, \abs{Q}^{m}
    \\
    &\leq c\, M_Q \big(\phi_a(f)\big) + c\, \abs{Q}^{m}.
  \end{align*}
  This proves the claim.

  {\bf Case: $M_Q f \leq a \leq \abs{Q}^m$}
  
  Using Corollary~\ref{cor:shift_ch}, the key estimate
  Theorem~\ref{thm:jensenpx}, and again Corollary~\ref{cor:shift_ch} we
  estimate
  \begin{align*}
    \phi_a(x,M_Q f) &\leq c\, \phi(x,M_Q f) + c\, \phi(x,a)
    \\
    &\leq c\, \phi(x,a)
    \\
    &\leq c\, a \phi(x,1)
    \\
    &\leq c\, \abs{Q}^m.
  \end{align*}
  This proves the claim.

  {\bf Case:  $M_Q f \leq a$ and $\abs{Q}^m \leq a$}

  Since $M_Q f \leq a$, we have
  \begin{align*}
    \phi_a(x,M_Q f) &\leq c\, \phi''(a) (M_Q f)^2.
  \end{align*}
  Recall that $\abs{Q}^m \leq a \leq \abs{Q}^{-m}$. We choose $x_Q^- \in
  Q$ such that $p(x_Q^-) = p_Q^-$. Now we conclude with
  Lemma~\ref{lem:pxpy} that
  \begin{align*}
    \phi''(a) &\leq c\,\phi''(x_Q^-, a) \leq c\,\phi''(x_Q^-, a+M_Q f)
  \end{align*}
  using in the last step that $M_Q f \leq a$. This and the previous
  estimate imply
  \begin{align*}
    \phi_a(x,M_Q f) &\leq c\, \phi''(x_Q^- a+M_Q f) (M_Q f)^2 \leq c\,
    \phi_a(x_Q^-,M_Q f)
    \\
    \intertext{We can now apply Jensen's inequality for the N-function
      $\phi_a(x_Q^-,\cdot)$ to conclude} \phi_a(x,M_Q f) &\leq
    c\, M_Q\big(\phi_a(x_Q^-,f)\big).
  \end{align*}
  Now, the claim follows if we can prove
  \begin{align}
    \label{eq:aux1}
    \phi_a(x_Q^-,f) \leq c\, \phi_a(x,f) \qquad \text{point wise on
      ~$Q$}.
  \end{align}
  So let $x \in Q$. If
  $\abs{f(x)}\leq a$, then with Lemma~\ref{lem:pxpy} and $\abs{Q}^m
  \leq a + \abs{f(x)} \leq 2\, \abs{Q}^{-m}$ we have
  \begin{align*}
    \phi_a(x_Q^-,\abs{f(x)}) &\leq c\,
    \phi_a''(x_Q^-,a+\abs{f(x)})\, \abs{f(x)}^2
    \\
    &\leq c\, \phi''(x,a+\abs{f(x)})\, \abs{f(x)}^2
    \\
    &\leq c\, \phi_a(x,\abs{f(x)}).
  \end{align*}
  It $\abs{f(x)} \geq a$, then
  \begin{align*}
    \phi_a(x_Q^-, \abs{f(x)}) &\leq c\,\phi_a(x,\abs{f(x)})\,
    \frac{\phi''(x_Q^-, a+\abs{f(x)})}{\phi''(x,a+\abs{f(x)})}
    \\
    &\leq c\,\phi_a(x,\abs{f(x)})\, (\kappa +
    a+\abs{f(x)})^{p_Q^- - p(x)}
    \\
    &\leq c\,\phi_a(x,\abs{f(x)})\, a^{p_Q^- - p(x)}.
    \\
    &\leq c\,\phi_a(x,\abs{f(x)}),
  \end{align*}
  where we used Lemma~\ref{lem:pxpy} in the last step and $\abs{Q}^m
  \leq a \leq \abs{Q}^{-m}$. Therefore \eqref{eq:aux1} folds, which concludes the proof.
\end{proof}

\subsection{\Poincare{} type estimates}
\label{ssec:Poincare}

We will show now that the shifted key estimate of
Theorem~\ref{thm:jensenpxshift} implies the following \Poincare{} type
estimate.
\begin{theorem}
  \label{thm:poincareshift}
  Let $p \in \PPln(\Rn)$ wit $p^+<\infty$. Then for every $m>0$ there
  exists $c>0$ only depending on $m$ and $c_{\log}(p)$ and $p^+$
  such that
  \begin{align*}
    \int_Q \phi_a\bigg(x,
    \frac{\abs{u(x)-\mean{u}_Q}}{\ell(Q)} \bigg)\,dx&\leq c\,
    \int_Q \phi_a(x,\abs{\nabla u(x)})\,dx + c\, \abs{Q}^{m}
  \end{align*}
  for every cube (or ball) $Q \subset \Rn$ with $\ell(Q) \leq 1$ and
  all $u\in W^{1,\px}(Q)$ with
  \begin{align*}
    a+\dashint_Q \abs{\nabla u}\,dy \leq \max \set{1, \abs{Q}^{-m}} =
    \abs{Q}^{-m}. 
  \end{align*}
\end{theorem}
\begin{proof}
  The proof is similar to the one of Proposition~8.2.8
  of~\cite{DieHHR11}. We can assume $m \ge 1$.

  By $\mathcal{W}_k$ we denote the family of dyadic cubes of size
  $2^{-k}$, i.e.  $2^{-k} ((0,1)^n + \overline{k})$ with $k \in \setZ$
  and $\overline{k} \in \setZ^n$. Let us fix $k_0\in\setZ$ such that
  $2^{-k_0-1}\leq \ell(Q)\leq 2^{-k_0} \le 1$. Then as in the proof of
  Proposition~8.2.8 of~\cite{DieHHR11} we have
  \begin{align*}
    |u(x) - \mean{u}_Q| \leq c \int_Q \frac{|\nabla
      u(y)|}{|x-y|^{n-1}} dy \leq c\, \ell(Q) \sum_{k=k_0+2}^\infty 2^{-k}
    T_k \big(\chi_Q |\nabla u|\big) (x),
  \end{align*}
  where the {\em averaging operator $T_k$} is given by
  \begin{align*}
    T_k f &:= \sum_{W \in \mathcal{W}_k} \chi_W M_{2W} f
  \end{align*}
  for all $f \in L^1_{\loc}(\Rn)$, where $2W$ is the cube with the
  same center as~$W$ but twice its diameter.

  We estimate by convexity and the locally finiteness of the
  family~$\mathcal{W}_k$
  \begin{align*}
    I &:= \int_Q \phi_{a}\bigg(x,
    \frac{\abs{u(x)-\mean{u}_Q}}{\ell(Q)} \bigg)\,dx
    \\
    &\leq c\, \sum_{k=k_0}^\infty 2^{-k} \int_Q \phi_a \bigg(x,
    \sum_{W \in \mathcal{W}_k} \chi_W(x) M_{2W} (\chi_Q \abs{\nabla
      u}) \bigg) \,dx
    \\
    &\leq c \sum_{k=k_0}^\infty 2^{-k}\sum_{W \in
      \mathcal{W}_k} \int_{W} \phi_a \big(x, M_{2W}
    (\chi_Q \abs{\nabla u}) \big) \,dx
    \\
    &\leq c \sum_{k=k_0}^\infty 2^{-k}\sum_{W \in
      \mathcal{W}_k} \int_{2W} \phi_a \big(x, M_{2W}
    (\chi_Q \abs{\nabla u}) \big) \,dx.
  \end{align*}
  Using the assumption on $u$, $\ell(2W) \leq 2^{-k_0-1} \le
  \ell(Q) \le 1$ and $m \ge 1$ we estimate
  \begin{align*}
    a+\dashint_{2W} \chi_Q \abs{\nabla u}\,dy \leq 
    a+ \frac{\abs{Q}}{\abs{2W}} \dashint_Q \abs{\nabla u}\,dy 
    \le \frac{\abs{Q}}{\abs{2W}} \bigg( a + \dashint_Q \abs{\nabla
      u}\,dx \bigg) \le
    \max \set{1, \abs{2W}^{-m}}.
  \end{align*}
  In particular, this and $\ell(2W) \le 1$ allow to
  apply the shifted key estimate of
  Theorem~\ref{thm:jensenpxshift}. This yields
  \begin{align*}
    I &\leq c \sum_{k=k_0}^\infty 2^{-k}\sum_{W \in \mathcal{W}_k}
    \int_{2W}\bigg(\dashint_{2W} \chi_Q(\phi_a (y, \abs{\nabla
      u}) \,dy+\abs{W}^m\bigg)\,dx
    \\
    &\leq c \sum_{k=k_0}^\infty 2^{-k}\sum_{W \in \mathcal{W}_k}
    \bigg(\int_{2W} \chi_Q \phi_a(y, \abs{\nabla u})
    \,dy+\abs{Q}^m\bigg)
    \\
    &\leq c \int_Q\phi_a (y,\abs{\nabla u}) \,dy+\abs{Q}^m
  \end{align*}
\end{proof}

\section{Interpolation and variable exponent spaces}
\label{sec:interpolation-wk-phi}
\noindent
Let $\Omega \subset \setR^n$ be a connected, open (possibly unbounded)
domain with polyhedral boundary. We assume that $\partial \Omega$ is
Lipschitz continuous. For an open, bounded (non-empty) set $U \subset
\setR^n$ we denote by $\abs{U}$ its $n$-dimensional Lebesgue measure,
by $h_U$ the diameter of $U$, and by $\rho_U$ the supremum of the
diameters of inscribed balls.  For $f \in L^1_\loc(\setR^n)$ we define
\begin{align*}
  \dashint_U f(x)\,dx &:= \frac{1}{\abs{U}} \int_U f(x)\,dx.
\end{align*}
For a finite set $A$ we define $\# A$ to be the number of elements of
$A$.  We write $f\sim g$ iff there exist constants $c, C>0$, such
that
\begin{align*}
  c\,f &\le g \le C\,f\,,
\end{align*}
where we always indicate on what the constants may depend.
Furthermore, we use $c$ as a generic constant, i.e.\ its
value my change from line to line but does not depend on the important
variables.

Let $\mathcal{T}_h$ be a simplicial subdivision of $\Omega$. By
\begin{align*}
  h &:= \max_{K \in \mathcal{T}_h} h_K
\end{align*}
we denote the maximum mesh size. We assume that $\mathcal{T}_h$ is
non-degenerate:
\begin{align}
  \label{eq:nondeg}
  \max_{K \in \mathcal{T}_h} \frac{h_K}{\rho_K} \leq \gamma_0.
\end{align}
For $K \in \mathcal{T}_h$ we define the set of neighbors $N_K$ and
the neighborhood $S_K$ by
\begin{align*}
  N_K &:= \set{K' \in \mathcal{T}_h \,:\, \overline{K'} \cap
    \overline{K} \not= \emptyset},
  \\
  S_K &:= \text{interior} \bigcup_{K' \in N_K} \overline{K'}.
\end{align*}
Note that for all $K,K' \in \mathcal{T}_h$: $K' \subset
\overline{S_{K}} \Leftrightarrow K \subset \overline{S_{K'}}
\Leftrightarrow \overline{K} \cap \overline{K'} \not=\emptyset$. Due
to our assumption on $\Omega$ the $S_K$ are connected, open bounded sets. 

It is easy to see that the non-degeneracy~\eqref{eq:nondeg} of
$\mathcal{T}_h$ implies the following properties, where the constants
are independent of $h$:
\begin{enumerate}
\item \label{mesh:SK} $\abs{S_K} \sim \abs{K}$ for all $K \in
  \mathcal{T}_h$.
\item \label{mesh:NK} There exists $m_1 \in \setN$ such that $\# N_K
  \leq m_1$.
\end{enumerate}
  
For $G\subset \setR ^n$ and $m\in \setN _0$ we denote by
$\mathcal{P}_m(G)$ the polynomials on $G$ of degree less than or equal
to $m$. Moreover, we set $\mathcal{P}_{-1}(G):=\set {0}$. Let us
characterize the finite element space $V_h$ as
\begin{align}\label{def:Vh}
  V_h &:= \set{v \in L^1_\loc(\Omega)\,:\, v|_K
    \in \mathcal{P}_K},
\end{align}
where
\begin{align}
  \label{eq:Pr0}
  \mathcal{P}_{r_0} (K) \subset \mathcal{P}_K \subset
  \mathcal{P}_{r_1}(K)
\end{align}
for given $r_0 \leq r_1 \in \setN_0$. 
Since $r_1
\in \setN_0$ there 
exists a constant $c=c(r_1)$ such that for all $\bfv_h \in (V_h)^N$, $K
\in \mathcal{T}_h$, $j \in \setN_0$, and $x\in K$ holds
\begin{align}
  \label{eq:pntwmean}
  h^j\abs{\nabla^j \bfv_h(x)} &\leq c\, h^j\dashint_{K}
  \abs{\nabla^j \bfv_h(y)} \,dy\leq c\dashint_K\abs{\bfv_h(y)}\,dy,
\end{align}
where we used \eqref{eq:nondeg}. Here and in the remainder of the
paper gradients of functions from $ V_h$ are always understood in a
local sense, i.e.~on each simplex $K$ it is the pointwise gradient of
the local polynomial.

We will show in this section how the classical results for the
interpolation error generalizes to the setting of Sobolev spaces with
variable exponents.  Instead, of deriving estimates for a specific
interpolation operator, we will deduce our results just from some
general assumptions on the operator. Note that e.g.~the Scott-Zhang
operator~\cite{ScoZha90} satisfies all our requirements.  More
precisely, we assume the following.
\begin{assumption}
  \label{ass:intop}
  Let $l_0 \in \setN_0$ and let $\Pi_h \,:\, (W^{l_0,1}(\Omega))^N \to
  (V_h)^N$.
  \begin{enumerate}
  \item For some $l \geq l_0$ and $m_0 \in \setN_0$ holds uniformly in $K \in
    \mathcal{T}_h$ and $\bfv \in (W^{l,1}(\Omega))^N$
  \begin{align}
    \label{eq:stab}
    \sum_{j=0}^{m_0} \dashint_K \abs{h_K^j \nabla^j \Pi_h \bfv}\,dx &\leq
    c(m_0,l)\, \sum_{k=0}^l h_K^{k} \dashint_{S_K} \abs{\nabla^k
      \bfv}\,dx.
  \end{align}
\item For all $\bfv \in (\mathcal{P}_{r_0})^N(\Omega)$ holds
  \begin{align}
    \label{eq:proj}
    \Pi_h \bfv &= \bfv.
  \end{align}
  \end{enumerate}
\end{assumption}

Note that the constant in \eqref{eq:stab} will depend on the
non-degeneracy constant $\gamma_0$ of the mesh. In this way all the
results below will depend on~$\gamma_0$.

\begin{remark}
  The property \eqref{eq:stab} is called {\em
    $W^{l,1}(\Omega)$-stability} of the interpolation operator
  $\Pi_h$. Because of \eqref{eq:pntwmean} we can choose $m_0$ to be $0$ with no loss of generality.

  In many cases, e.g.~for the \Clement{} and the Scott-Zhang
  interpolation operator, we have $\Pi_h \bfv := (\Pi_h v_1,
  \ldots,\Pi_h v_N)$. Both operators satisfy
  Assumption~\ref{ass:intop} (cf.~\cite{Cle75,ScoZha90}). In
  fact,   the scaling invariant formulation \eqref{eq:stab} can be
  easily derived from the proofs there. Note, that for the Scott-Zhang
  interpolation operator holds a much stronger property, namely
  \begin{align}
    \label{eq:proj1}
    \Pi_h \bfv &= \bfv, \qquad \forall \bfv \in (V_h)^N.
  \end{align}
\end{remark}

\begin{remark}
  The Scott-Zhang interpolation operator is defined in such a way that
  it preserves homogeneous boundary conditions, i.e.
  \begin{align*}
    \Pi_h \,:\,W^{1,1}_0(\Omega) \to  V \cap W^{1,1}_{0}.
  \end{align*}
  Thus we have to choose in this case $l_0=1$ in \eqref{eq:stab}.
  However, there is a version of the Scott-Zhang interpolation
  operator which does not preserve boundary values (cf.~remark
  after~(4.6) in~\cite{ScoZha90}). In this case we can choose $l_0 =0$
  in \eqref{eq:stab}.
\end{remark}

Now, we will deduce solely from the Assumption~\ref{ass:intop} that
the same operator $\Pi_h$ is also a good interpolation operator for
the generalized Sobolev spaces $(W^{l,p(\cdot)}(\Omega))^N$. We begin with the
stability. 

\begin{lemma}[\bf Stability]
  \label{thm:stabo}
  Let $\Pi_h$ and $l$ satisfy Assumption~\ref{ass:intop}, $m>0$ and
  let $p\in\PPln(\Omega)$ with $p^+<\infty$. Then for all $j \in
  \setN_0$, $K \in \mathcal{T}_h$ and $\bfv \in (W^{l,p(\cdot)})^N$
  with
  \begin{align*}
    \max_{k=0,\dots,l} \dashint_{S_K} h_K^k \abs{\nabla^k \bfv}\,dy
    \leq c_2 \max \set{1, \abs{K}^{-m}}.
  \end{align*}
  holds
    \begin{align}
      \label{eq:stabo}
      \dashint_K \phi_a\big(\cdot,\abs{h_K^j \nabla^j \Pi_h
        \bfv}\big)\,dx &\leq c\, \sum_{k=0}^l
      \dashint_{S_K}\phi_a\big(\cdot,\abs{h_K^k\nabla^k
        \bfv}\big)\,dx+c\,h_K^m,
    \end{align}
    where $c = c(k,l,m, c_{\log},p^+,c_2)$.
\end{lemma}
\begin{proof}
  Since $p^+<\infty$ the $\Delta_2$-constant of $\phi_a$ is uniformly
  bounded with respect to~$a$ and it therefore suffices to consider
  the case $c_2=1$.  Using ~\ref{ass:intop} we gain
  \begin{align*}
    \dashint_K \phi_a\big(\cdot,\abs{h_K^j \nabla^j \Pi_h
      \bfv}\big)\,dx &\leq c \dashint_K \phi_a\bigg(\cdot,
    \dashint_K h_K^j \abs{\nabla^j \Pi_h \bfv(y)}\,dy \bigg) \,dx
    \\
    &\leq c \dashint_K \phi_a\bigg(\cdot, \sum_{k=0}^l
    h_K^{k} \dashint_{S_K} \abs{\nabla^k \bfv(y)}\,dy \bigg) \,dx.
  \end{align*}
  Now with Theorem \ref{thm:jensenpxshift} as well as the convexity
  and $\Delta_2$-condition of $\varphi_a$ we find
  \begin{align*}
    \dashint_K \phi_a\bigg(\cdot,\dashint_{S_K}h_K^j \abs{\nabla^j
      \Pi_h\bfv}\,dy \bigg)\,dx &\leq c
    \sum_{k=0}^l\dashint_K \phi_a\bigg(\cdot, h_K^{k}
    \dashint_{S_K} \abs{\nabla^k \bfv(y)}\,dy \bigg) \,dx\\
    &\leq c \sum_{k=0}^l \dashint_K\dashint_{S_K}
    \phi_a\big(\cdot,h_K^j \abs{\nabla^j \bfv}
    \big)\,dy\,dx +ch_K^m\\
    &\leq c\sum_{k=0}^l \dashint_{S_K} \phi_a\big(\cdot,h_K^j
    \abs{\nabla^j \bfv}\,dy \big)\,dx +ch_K^m.
  \end{align*}
  This proves the theorem.
\end{proof}

\begin{lemma}[\bf Approximability]
  \label{thm:oappr}
  Let $\Pi_h$ and $l$ satisfy Assumption~\ref{ass:intop} with $l \leq
  r_0 +1$ and let $p\in\PPln(\Omega)$ with $p^+<\infty$ and let $m>0$.
  Then for all $j\leq l$, $K \in \mathcal{T}_h$ and $\bfv \in
  (W^{l,p(\cdot)})^N$
  with
  \begin{align*}
    \dashint_{S_K} h_K^l \abs{\nabla^l \bfv}\,dy \leq c_3 \max \set{1,
      \abs{K}^{-m}}.
  \end{align*}
  holds
  \begin{align}
    \label{eq:oappr1}
    \dashint_K \phi_a\big(\cdot, h_K^{j}|\nabla^j (\bfv - \Pi_h \bfv)
    |\big)\,dx &\leq c\,\dashint_{S_K} \phi_a\big(\cdot,h_K^{l} \big|
    \nabla^l \bfv \big|\big)\,dx+ch_K^m,
  \end{align}
  where $c = c(k,l,m, c_{\log},p^+,c_3)$.
\end{lemma}
\begin{proof}
  We fix $\frp \in (\mathcal{P}_{l-1})^N$ such that $\int_{S_K}
  \nabla^k (\bfv - \frp)\,dx = 0$ for all $k=0,\dots,l-1$. Then by
  \Poincare{}
  \begin{align}
    \label{eq:kleingenug}
    \max_{k=0,\dots,l} \dashint_{S_K} h_K^k \abs{\nabla^k(\bfv -
      \bfp)}\,dx \le c\, \dashint_{S_K} h_K^l \abs{\nabla^l \bfv}\,dx
    \le c \max \set{1, \abs{K}^{-m}}.
  \end{align}
  Due to~\eqref{eq:proj} and $l-1 \le r_0$ we have $\Pi_h \frp =
  \frp$.  From this, the convexity of~$\phi_a$, and
  $\Delta_2(\phi_a)<\infty$ (uniformly with respect to~$a$) we obtain
  for all $0 \leq j \leq l$
  \begin{align*}
    \lefteqn{I := \dashint_K \phi_a\big(\cdot, h_K^{j}|\nabla^j (\bfv -
        \Pi_h \bfv)| \big)\,dx} \hspace{5mm} &
    \\
    &\leq c\, \dashint_K \phi_a\big(\cdot, h_K^{j}|\nabla^j (\bfv -
      \frp)| \big)\,dx + c\, \dashint_K  \phi_a\big(\cdot, h_K^{j}
    |\nabla^j (\Pi_h (\bfv-\frp))| \big)\,dx.
  \end{align*}
  Now, we use Lemma~\ref{thm:stabo} for the functions $(\bfv -
  \frp)$ to obtain
  \begin{align*}
    I &\leq c\,\dashint_K \phi_a\big(\cdot, h_K^{j}| \nabla^j (\bfv -
    \frp)| \big)\,dx+c\sum_{k=0}^l \dashint_{S_K} \phi_a\big(\cdot,
    h_K^{k} \abs{\nabla^k( \bfv-\frp)} \big) \,dx+ch_K^m
    \\
    &\leq c\sum_{k=0}^l \dashint_{S_K} \phi_a\big(\cdot, h_K^{k}
    \abs{\nabla^k( \bfv-\frp)} \big) \,dx+ch_K^m
  \end{align*}
  Due to~\eqref{eq:kleingenug} we can apply repeatedly the
  Poincar{\'e}-type inequality from Theorem~\ref{thm:poincareshift} and gain
  \begin{align*}
    \dashint_K \phi_a\big(\cdot, h_K^{j}|\nabla^j (\bfv - \Pi_h \bfv)|
    \big)\,dx\leq c\dashint_{S_K} \phi_a\big(\cdot,
    h_K^l\abs{\nabla^l\bfv} \big) \,dx+ch_K^m,
  \end{align*}
  where we also use the uniform $\Delta_2$-constants of~$\phi_a$ with
  respect to~$a$.
\end{proof}

\begin{corollary}[\bf Continuity]
  \label{cor:ocont}
  Under the assumptions of Lemma~\ref{thm:oappr} holds
  \begin{align}
    \label{eq:ocont1}
    \dashint_K \phi_a\big(\cdot,h_K^l|\nabla^l \Pi_h \bfv|\big)\,dx
    &\leq c \dashint_{S_K} \phi_a\big(\cdot,h_K^l|\nabla^l
    \bfv|\big)\,dx+ch_K^m,
  \end{align}
  with $c=c = c(l,m, c_{\log},p^+,\gamma_0,c_3)$.
\end{corollary}
\begin{proof}
  The assertion follows directly from~\eqref{eq:oappr1}, $j=l$, and
  the triangle inequality.
\end{proof}

\section{Convergence}
\label{sec:conv}
\noindent
Let us now apply the previous results to a finite element
approximation of the system \eqref{eq:app_sys}
considered on a bounded Lipschitz domain $\Omega$ and equipped with
zero Dirichlet boundary conditions, where $p\in \PPln (\Omega)$. Due
to these assumptions it is easy to show, with the help of the theory
of monotone operators, that there exists a solution $\bfv \in (W^{1,
  \px}_{0}(\Omega))^N$ such that for all $\bfpsi \in
(W^{1,\px}_{0}(\Omega))^N $ holds
\begin{align}
  \label{eq:weak}
  \int_\Omega \bfA (\cdot,\nabla \bfv) \cdot\nabla\bfpsi\, dx =
  \int_\Omega \bff \cdot \bfpsi\, dx,
\end{align}
where $\bfA(x,\bfxi):=(\kappa+|\bfxi|)^{p(x)-2}\bfxi$.

Let $V_h$ be given as in \eqref{def:Vh} with $r_0 \ge 1$ such that
$V_h \subseteq  W^{1,\px}(\Omega)$. Denote
\begin{align*}
  \bfV_{h,0}:= \big (V_h\cap
  W^{1,\px}_{0}(\Omega)\big )^N .
\end{align*}
Thus we now restrict ourselves to conforming finite element spaces
containing at least linear polynomials.  The finite
element approximation of \eqref{eq:weak} reads as follows: find $\bfv_h
\in \bfV_{h,0}$  such that for all $\bfpsi _h \in \bfV_{h,0} $ holds 
\begin{align}
  \label{eq:weak_h}
  \int_\Omega \bfA (\cdot,\nabla\bfv_h) \cdot \nabla\bfpsi_h\, dx =
  \int_\Omega \bff \cdot \bfpsi_h\, dx.
\end{align}
Again it is easy to see that this problem has a unique solution. 

In the following we define
\begin{align*}
  \bfF(x,\bfxi)&:=(\kappa+|\bfxi|)^{\frac{p(x)-2}{2}}\bfxi\text{ for }\kappa\in[0,1]
\end{align*}
We begin with a best approximation result.
\begin{lemma}[best approximation]
  \label{lem:bestappr}
  Let $p\in\PPln (\Omega)$ with $p^{-}>1$ and $p^+<\infty$. Let $\bfv
  \in (W^{1, \px}_0(\Omega))^N$ and $\bfv_h \in \bfV_{h,0}$ be
  solutions of \eqref{eq:weak} and \eqref{eq:weak_h} respectively.
  Then there exists $c>0$ depending only on $c_{\log}$, $p^-$, $p^+$
  and $\gamma_0$ such that
  \begin{align}
    \label{eq:2}
    \norm{\bfF(\cdot,\nabla\bfv) - \bfF(\cdot,\nabla\bfv_h)}_2 \le c
    \, \min _{\bfpsi_h \in \bfV_{h,0}} \norm{\bfF(\cdot,\nabla \bfv) -
      \bfF(\cdot,\nabla \bfpsi_h)}_2 \,.
  \end{align}
\end{lemma}

\begin{proof}
  We use Lemma \ref{lem:hammer} and the equations
  \eqref{eq:weak} and \eqref{eq:weak_h} to find 
  \begin{align*}
\int_\Omega &\abs{\bfF(\cdot,\nabla\bfv) - \bfF(\cdot,\nabla\bfv_h)}^2\, dx\\
   &= c\, \int_\Omega\! \big (\bfA (\cdot,\nabla\bfv) \!-\! \bfA
    (\cdot,\nabla\bfv_h)\big ): \big (\nabla \bfv \!-\!\nabla\bfv_h
    \big )\, dx
    \\
    &= c\, \int_\Omega \big (\bfA (\cdot,\nabla\bfv) \!-\! \bfA
    (\cdot,\nabla \bfv_h)\big ): \big (\nabla\bfv \!-\!\nabla\bfpsi_h
    \big )\, dx
    \\
    &\leq c\, \int_\Omega 
(\phi_{\abs{\nabla \bfv}+\kappa})'(x,\abs{\nabla\bfv-\nabla\bfv_h})
    \abs{\nabla\bfv -\nabla\bfpsi_h}\, dx.
  \end{align*}
for all $\bfpsi _h \in \bfV_{h,0}$.
  Now, Young's inequality, see Lemma~\ref{lem:young}, for
  $\phi_{\abs{\nabla \bfv}+\kappa}(x,\cdot)$ implies for $\tau>0$
  \begin{align*}
    \int_\Omega &\abs{\bfF(\cdot,\nabla\bfv) - \bfF(\cdot,\nabla\bfv_h)}^2\, dx
\\
    \le &\tau\int_\Omega (\phi_{\abs{\nabla \bfv}+\kappa})(x,\abs{\nabla\bfv-\nabla\bfv_h})\, dx
\\&\quad +c(\tau)\int_\Omega  (\phi_{\abs{\nabla \bfv}+\kappa})(x,\abs{\nabla\bfv-\nabla\bfpsi_h})\, dx.
\\
\sim &\,\tau\int_\Omega \abs{\bfF(\cdot,\nabla\bfv) - \bfF(\cdot,\nabla\bfv_h)}^2\, dx
  \\
&\quad +c(\tau)\int_\Omega\abs{\bfF(\cdot,\nabla\bfv) - \bfF(\cdot,\nabla\bfpsi_h)}^2\, dx,
  \end{align*}
where we used Lemma~\ref{lem:hammer} in the last step. Hence, for $\tau$ small enough we can absorb the $\tau$-integral in the l.h.s. This implies
  \begin{align*}
    \norm{\bfF(\cdot,\nabla\bfv) - \bfF(\cdot,\nabla\bfv_h)}_2^2 &\le
    c \, \norm{\bfF(\cdot,\nabla\bfv) -
      \bfF(\cdot,\nabla\bfpsi_h)}_2^2 .
  \end{align*}
  Taking the minimum over all $\bfpsi _h \in \bfV_{h,0}$ proves the
  claim.
\end{proof}
\begin{lemma}
  \label{lem:app_V1}
  Let $p\in\PPln (\Omega)$ with $p^{-}>1$ and $p^+<\infty$ and let
  $m>0$.  Let $\Pi_h$ satisfy Assumption~\ref{ass:intop} with $l=1$
  and $\Pi_h \frq = \frq$ for all $\frq \in
  (\mathcal{P}_1)^N(\Omega)$. (This is e.g.~satisfied if $r_0 \geq
  1$.) Let $\bfv \in (W^{1,\px}(\Omega))^N$ then for all $K \in
  \mathcal{T}_h$ with $h_K \le 1$ and all $\bfQ \in \setR^{N \times
    n}$  with
  \begin{align*}
    \abs{\bfQ} + \dashint_{S_K} \abs{\nabla \bfv}\,dx \le c_4\, \max
    \set{1, \abs{K}^{-m}},
  \end{align*}
  it holds that
  \begin{align}
    \label{eq:app_V1}\begin{aligned}
      \dashint_K \bigabs{\bfF (\cdot,\nabla\bfv) - \bfF
        (\cdot,\nabla\Pi_h \bfv)}^2 \,dx
      \leq c\,  \dashint_{S_K}
      \bigabs{\bfF(\cdot,\nabla \bfv) - \bfF(\cdot,\bfQ)}^2
      \,dx+ch_K^m,
    \end{aligned}
  \end{align}
  with $c=c(c_{\log},p^-,p^+,m,\gamma_0,c_4)$.
\end{lemma}
\begin{proof}
  Note that $\bfv \in (W^{1,\px}(\Omega))^N$ implies
  $\bfF(\cdot,\nabla\bfv) \in (L^2(\Omega))^{N\times n}$.  For
  arbitrary $\bfQ \in \setR^{N \times n}$ we have
  \begin{align*}
    \lefteqn{\dashint_K \abs{\bfF (\cdot,\nabla \bfv) -
        \bfF(\cdot,\nabla \Pi_h \bfv)}^2 \,dx} \hspace{5mm}
    \\
    &\leq c\, \dashint_K \abs{\bfF(\cdot,\nabla \bfv) -
      \bfF(\cdot,\bfQ)}^2 \,dx + c\, \dashint_K \abs{\bfF
      (\cdot,\nabla \Pi_h \bfv) - \bfF(\cdot,\bfQ)}^2 \,dx.
    \\
    &=: (I) + (II).
  \end{align*}
  Let $\frp \in (\mathcal{P}_1)^N(S_K)$ be such that $\nabla \frp =
  \bfQ$. Due to $\Pi_h \frp = \frp$ there holds $\bfQ = \nabla \frp =
  \nabla \Pi_h \frp$. We estimate by Lemma \ref{lem:hammer} as follows
  \begin{align*}
    (II) &\leq c\, \dashint_K
    \phi_{\abs{\bfQ}+\kappa}\big(\cdot,\abs{\nabla\Pi_h \bfv - \bfQ}
    \big)\,dx
    \\
    &= c\, \dashint_K \phi_{\abs{\bfQ}+\kappa}\big(\cdot,\abs{\nabla \Pi_h
      \bfv - \nabla \Pi_h \frp} \big) \,dx
    \\
    &= c\, \dashint_K \phi_{\abs{\bfQ}+\kappa}\big(\cdot,\abs{\nabla \Pi_h
      (\bfv -\frp) } \big)\,dx.
  \end{align*}
  We use the crucial Corollary~\ref{cor:ocont} for
  $\bfv -\frp$  and $a=\abs{\bfQ}+\kappa$ to obtain
  \begin{align*}
    (II) &\leq c\, \dashint_{S_K} \phi_{\abs{\bfQ}+\kappa}\big(\cdot,\abs{\nabla
      (\bfv -\frp) } \big)\,dx +ch^2_K= c\, \dashint_{S_K}
    \phi_{\abs{\bfQ}+\kappa}\big(\cdot,\abs{\nabla \bfv - \bfQ} \big)\,dx+ch_K^m.
  \end{align*}
  Now, with Lemma \ref{lem:hammer}
  \begin{align*}
    (II) &\leq c\, \dashint_{S_K} \abs{\bfF(\cdot,\nabla \bfv) -
      \bfF(\cdot,\bfQ)}^2 \,dx+ch_K^m.
  \end{align*}
  Since, $\abs{K} \sim \abs{S_K}$ and $K \subset S_K$ we also have
  \begin{align*}
    (I) &\leq c\, \dashint_{S_K} \abs{\bfF(\cdot,\nabla \bfv) -
      \bfF(\cdot,\bfQ)}^2 \,dx.
  \end{align*}
  Overall,
  \begin{align*}
    \dashint_K \bigabs{\bfF (\cdot,\nabla \bfv) -
      \bfF (\cdot,\nabla \Pi_h \bfv)}^2 \,dx &\leq c\,
    \dashint_{S_K} \bigabs{\bfF(\cdot,\nabla \bfv) -
      \bfF(\cdot,\bfQ)}^2 \,dx+ch_K^m.
  \end{align*}
  Taking the infimum over all $\bfQ$ proves~\eqref{eq:app_V1}. 
\end{proof}
\begin{lemma}
  \label{lem:app_V2}
  Let $p\in C^{0,\alpha}(\overline{\Omega})$ with $p^{-}>1$ and let
  $m>0$.  Let $\Pi_h$ satisfy Assumption~\ref{ass:intop} with $l=1$
  and $\Pi_h \frq = \frq$ for all $\frq \in
  (\mathcal{P}_1)^N(\Omega)$. (This is e.g.~satisfied if $r_0 \geq
  1$.) Let $\bfv \in (W^{1,\px}(\Omega))^N$ then for all $K \in
  \mathcal{T}_h$ with $h_K \le 1$ with
  \begin{align*}
    \dashint_{S_K} \abs{\nabla \bfv}\,dx \le c_5\, \max \set{1,
      \abs{K}^{-m}},
  \end{align*}
  it holds that
  \begin{align}
     \dashint_K \bigabs{\bfF (\cdot,\nabla\bfv) - \bfF
        (\cdot,\nabla\Pi_h \bfv)}^2 \,dx
&\leq
    c
    h_K^{2\alpha}\bigg(\dashint_{S_K}\ln(\kappa+|\nabla\bfv|)^2
    (\kappa+|\nabla\bfv|)^{\px}\,dx+1\bigg)\nonumber
    \\
    &+c\, h_K^{2}\bigg(\, \dashint_{S_K} \bigabs{\nabla
      \bfF(\cdot,\nabla\bfv)}^2\,dx\bigg) \label{eq:app_V2}
  \end{align}
  with
 $c=c(c_{\log},\norm{p}_{C^{0,\alpha}(\overline{\Omega})},\alpha,p^-,m,\gamma_0,c_5)$.
\end{lemma}
\begin{proof}
  In order to show \eqref{eq:app_V2} we need a special choice of
  $\bfQ$ in Lemma~\ref{lem:app_V1}.
  Precisely we want that
  \begin{align}\label{eq:Q}
    \dashint_{S_K}\bfF(\cdot,\bfQ)\,dx=\dashint_{S_K}\bfF(\cdot,\nabla\bfv)\,dx.
  \end{align}
  This is possible if the function 
  \begin{align*}
    \lambda(\bfQ):=\dashint_{S_K}\bfF(\cdot,\bfQ)\,dx=
    \dashint_{S_K}(\kappa+|\bfQ|)^{\frac{\px-2}{2}} \,dx\,\bfQ
  \end{align*}
  is surjective.  Obviously we can choose every direction in $\mathbb
  R^{N\times n}$ via the direction of $\bfQ$. On the other hand the
  modulus
  \begin{align*}
    |\lambda(\bfQ)|=\dashint_{S_K}(\kappa+|\bfQ|)^{\frac{\px-2}{2}}
    \abs{\bfQ}\,dx
  \end{align*}
  is an increasing continuous function of $|\bfQ|$ with
  $|\lambda(\bfzero)|\,=0$ and $|\lambda(\bfQ)|\rightarrow \infty$ for
  $|\bfQ|\rightarrow\infty$. Hence $\lambda:\mathbb R^{N\times
    n}\rightarrow\mathbb R^{N\times n}$ is surjective. Choosing $\bfQ$
  via \eqref{eq:Q} we have
  \begin{align*}
    \dashint_{S_K} &\bigabs{\bfF(\cdot,\nabla \bfv) -
      \bfF(\cdot,\bfQ)}^2 \,dx\\&\leq c \dashint_{S_K}
    \bigabs{\bfF(\cdot,\nabla \bfv) -
      \langle\bfF(\cdot,\nabla\bfv)\rangle_{S_K}}^2 \,dx+c
    \dashint_{S_K} \bigabs{\bfF(\cdot,\bfQ) -
      \langle\bfF(\cdot,\bfQ)\rangle_{S_K}}^2 \,dx.
  \end{align*}
  By Poincar{\'e}'s inequality we clearly gain
  \begin{align*}
    \dashint_{S_K} \bigabs{\bfF(\cdot,\nabla \bfv) -
      \langle\bfF(\cdot,\nabla\bfv)\rangle_{S_K}}^2 \,dx\leq
    \,ch_K^2\dashint_{S_K} \bigabs{\nabla\bfF(\cdot,\nabla \bfv) }^2
    \,dx
  \end{align*}
  We estimate
  \begin{align*}
    \dashint_{S_K} \bigabs{\bfF(\cdot,\bfQ) -
      \langle\bfF(\cdot,\bfQ)\rangle_{S_K}}^2 \,dx&\leq c
    \dashint_{S_K} \dashint_{S_K} \bigabs{\bfF(x,\bfQ) -
      \bfF(y,\bfQ)}^2 \,dx\,dy.
  \end{align*}
  For $x,y \in S_K$ we estimate
  \begin{align}
    \label{eq:diff_F}
    \lefteqn{\bigabs{\bfF(x,\bfQ) - \bfF(y,\bfQ)}} \qquad \nonumber &
    \\
    &\leq c\, \abs{p(x) - p(y)} \ln
    (\kappa\!+\!\abs{\bfQ})\Big(\big(\kappa \!+\!
    \abs{\bfQ}\big)^{\frac{p(x)-2}{2}} \!+\! \big(\kappa\!+\!
    \abs{\bfQ}\big)^{\frac{p(y)-2}{2}}\Big) \abs{\bfQ}.
  \end{align}
  This, $p \in C^{0,\alpha}(\overline{\Omega})$ and the previous
  estimate imply
  \begin{align}\label{eq:neu2}
    \dashint_{S_K} \bigabs{\bfF(\cdot,\bfQ) -
      \langle\bfF(\cdot,\bfQ)\rangle_{S_K}}^2 \,dx &\leq c\,
    h_K^{2\alpha}\bigg(\dashint_{S_K}\ln(\kappa+|\bfQ|)^2
    (\kappa+|\bfQ|)^{p(x)}\,dx + 1 \bigg).
  \end{align}
  In order to complete the proof of \eqref{eq:app_V2} we need to replace
  $\bfQ$ by $\nabla\bfv$. The choice of $\bfQ$ in \eqref{eq:Q} yields
  \begin{align}
    \abs{\bfQ}^{p_{S_K}^-}&\leq
    c\bigg(\biggabs{\dashint_{S_K}\bfF(\cdot,\bfQ)\,dx}^2+1\bigg)=
    c\bigg(\biggabs{\dashint_{S_K}\bfF(\cdot,\nabla\bfv)\,dx}^2
    +1\bigg)\label{eq:neu}
    \\
    &\leq c\bigg(\dashint_{S_K}\abs{\bfF(\cdot,\nabla\bfv)}^2\,dx+
    1\bigg)\leq \,c\,\abs{S_K}^{-1}\leq\,c\,h^{-n}.\nonumber
  \end{align}
  This allows to apply Lemma~\ref{lem:pxpy} to get
  \begin{align*}
    \dashint_{S_K}\ln(\kappa+|\bfQ|)^2
    (\kappa+|\bfQ|)^{p(x)}\,dx&\leq\,c\,\ln(1+|\bfQ|)^2
    (\kappa+|\bfQ|)^{p_{S_K}^{-}}
    \\
    &\leq\,c\,\Big(\Psi\Big(\abs{\bfQ}^{p_{S_K}^-}\Big)+1\Big),
  \end{align*}
  where $\Psi(t):=\ln(1+t) t$. Since $\Psi$ is convex we have as
  a consequence of \eqref{eq:neu} and Jensen's inequality
  \begin{align*}
    \Psi\Big(\abs{\bfQ}^{p_{S_K}^-}\Big) +1
    &\leq\,c\,\bigg(\dashint_{S_K}\Psi\big(\abs{\bfF(\cdot,\nabla\bfv)
    }\big)\,dx+1\bigg)
    \\
    &\leq\,c\,\bigg(\dashint_{S_K}\ln(\kappa+|\nabla\bfv|)^2
    (\kappa+|\nabla\bfv|)^{p(x)}\,dx +1\bigg).
  \end{align*}
  Inserting this in \eqref{eq:neu2} proves the claim.
\end{proof}

From Lemma~\ref{lem:bestappr}, Lemma~\ref{lem:app_V1} and
Lemma~\ref{lem:app_V2} we immediately deduce the following
interpolation error estimate.
\begin{theorem}
  \label{thm:error}
  Let $p\in C^{0,\alpha}(\overline{\Omega})$ with $p^{-}>1$. Let $\bfv
  \in (W^{1, \px}_0(\Omega))^N$ and $\bfv_h \in \bfV_{h,0}$ be
  solutions of \eqref{eq:weak} and \eqref{eq:weak_h} respectively.
  Then
  \begin{align*}
    \bignorm{\bfF (\cdot,\nabla\bfv) - \bfF (\cdot,\nabla \bfv_h)}_2
    &\leq c \, \min _{\bfpsi_h \in \bfV_{h,0}} \norm{\bfF(\cdot,\nabla
      \bfv) - \bfF(\cdot,\nabla \bfpsi_h)}_2
    \\
    &\leq c\, \bignorm{\bfF (\cdot,\nabla\bfv) - \bfF (\cdot,\nabla
      \Pi_h \bfv)}_2+c\, h^\alpha
    \\
    &\leq c\, h^{\alpha},
  \end{align*}
  where $c$ also depends on $\norm{\bfF(\nabla \bfv)}_{1,2}$.
\end{theorem}
\begin{remark}
  In the case of~$p \in C^{0,\alpha}$ with $\alpha <1$ one cannot
  expect that $\bfF(\cdot,\nabla \bfv) \in W^{1,2}(\Omega)$ even
  locally. However, motivated by the results of~\cite{DieE08} we
  believe that $p \in C^{0,\alpha}$ should imply $\bfF(\cdot , \nabla
  \bfv) \in \mathcal{N}^{\alpha,2}$, where $\mathcal{N}^{\alpha,2}$ is
  the \Nikolskii{} space with order of
  differentiability~$\alpha$. This is enough to ensure that
  \begin{align*}
    \sum_K \int_{S_K} \bigabs{\bfF(\cdot,\nabla \bfv) -
      \langle\bfF(\cdot,\nabla\bfv)\rangle_{S_K}}^2 \,dx\leq
    c\,h_K^{2\alpha}.
  \end{align*}
  This still implies the $\mathcal{O}(h^\alpha)$-convergence of
  Theorem~\ref{thm:error}. 
\end{remark}

\subsection{Frozen exponents}

For the application it is convenient to replace the exponent
$p(\cdot)$ by some local approximation. 
\begin{align*}
  p_{\mathcal{T}} &:= \sum_{K \in \mathcal{T}}  p(x_K)\chi_K,
\end{align*}
where $x_K:=\mathrm{argmin}_K p$, i.e. $p(x_K)= p_K^-$, and consider
\begin{align*}
  \bfA_{\mathcal T}(x,\bfxi)=\sum_{K\in\mathcal T}\chi_K(x)
  \bfA\big(x_K,\bfxi\big)
\end{align*}
instead of $\bfA$. In the following we will show that this does not
effect our results about convergence. Using monotone operator theory
we can find a function $\tilde{\bfv}_h\in\bfV_{h,0}$
\begin{align}
  \label{eq:weak_h'}
  \int_\Omega \bfA_{\mathcal T} (\cdot,\nabla\tilde{\bfv}_h) \cdot
  \nabla\bfpsi_h\, dx = \int_\Omega \bff \cdot \bfpsi_h\, dx
\end{align}
for all $\bfpsi\in \bfV_{h,0}$. The coercivity $\bfA_{\mathcal{T}}$
implies that $\tilde{\bfv}_h \in
(W^{1,p_{\mathcal{T}}(\cdot)}(\Omega))^N$. We will show now that also
$\tilde{\bfv}_h\in (W^{1,\px}(\Omega))^N$. We begin with the local estimate
\begin{align}
  \label{eq:v'1}
  \begin{aligned}
    \|\nabla\tilde{\bfv}_h\|_{L^\infty(K)}&\leq
    c \bigg(\dashint_{K}|\nabla\tilde{\bfv}_h |^{p(x_K)}\,dx\bigg)^{
      \frac{1}{p(x_K)}} \leq c\, h_K^{-\frac{d}{p(x_K)}}.
  \end{aligned}
\end{align}
Thus, we can apply Lemma~\ref{lem:pxpy} to find
\begin{align*}
  \int_\Omega |\nabla\tilde{\bfv}_h|^{\px} \,dx&\leq \int_\Omega
  (1+|\nabla\tilde{\bfv}_h|)^{\px} \,dx \leq c \int_\Omega
  (1+|\nabla\tilde{\bfv}_h|)^{p_{\mathcal{T}}} \,dx \leq c(\bff).
\end{align*}
\begin{remark}
  Indeed, the steps above show, that the norms
  $\norm{\cdot}_{p_{\mathcal{T}}(\cdot)}$ and
  $\norm{\cdot}_{p(\cdot)}$ are equivalent on any finite element space
  consisting locally of polynomials of finite order. Even more, if
  $\norm{g}_{\px} \leq 1$, then
  \begin{align*}
    \int_\Omega \abs{g(x)}^{p(x)}\,dx &\leq c\, \int_\Omega
    \abs{g(x)}^{p_{\mathcal{T}}(x)} \,dx + c\, h^m
  \end{align*}
  for any fixed~$m$, where $c=c(m)$. The same estimate holds also with
  $p$ and $p_{\mathcal{T}}$ interchanged.
\end{remark}
Following the same ideas as before we want to estimate
\begin{align*}
  \norm{\bfF_{\mathcal T}(\cdot,\nabla\bfv) - \bfF_{\mathcal
      T}(\cdot,\nabla\tilde{\bfv}_h)}_2^2 &=\sum_{K\in\mathcal
    T}\int_K\abs{\bfF(x_K,\nabla\bfv) -
    \bfF(x_K,\nabla\tilde{\bfv}_h)}^2\,dx,
  \\
  \text{where }\bfF_{\mathcal T}(x,\bfxi)&:=\sum_{K\in\mathcal T}\chi_K(x)
  \bfF\big(x_K,\bfxi\big).
\end{align*}
Now we are confronted with the problem that $\bfv$ and
$\tilde{\bfv}_h$ solve two different equations. Due to this an
additional error term occurs. Following the same approach as before we
gain for $\bfpsi\in\bfV_{0,h}$ arbitrary
\begin{align*}
  \norm{\bfF_{\mathcal T}(\cdot,\nabla\bfv) - \bfF_{\mathcal
      T}(\cdot,\nabla\tilde\bfv_h)}_2^2 &\leq c\,
  \int_\Omega\big(\bfA_{\mathcal T}(\cdot,\nabla
  \bfv)-\bfAT(\cdot,\nabla \tilde{\bfv}_h)\big):
  \big(\nabla\bfv-\nabla\tilde{\bfv}_h\big)\,dx
  \\
  &= c\, \int_\Omega\big(\bfA(\cdot,\nabla\bfv)-\bfAT(\cdot,
  \nabla\tilde{\bfv}_h)\big):
  \big(\nabla\bfv-\nabla\tilde{\bfpsi}_h\big)\,dx
  \\
  &\quad +c\,
  \int_\Omega\big(\bfAT(\cdot,\nabla\bfv)-\bfA(\cdot,\nabla\bfv)\big):
  \big(\nabla\bfv-\nabla\tilde{\bfv}_h\big)\,dx
  \\
  &= c\, \int_\Omega\big(\bfAT(\cdot,\nabla\bfv)-\bfAT(\cdot,
  \nabla\tilde{\bfv}_h)\big):
  \big(\nabla\bfv-\nabla\tilde{\bfpsi}_h\big)\,dx
  \\
  &\quad +c\,
  \int_\Omega\big(\bfAT(\cdot,\nabla\bfv)-\bfA(\cdot,\nabla\bfv)\big):
  \big(\nabla\bfv-\nabla\tilde{\bfv}_h\big)\,dx
  \\
  &\quad +c\,
  \int_\Omega\big(\bfA(\cdot,\nabla\bfv)-\bfAT(\cdot,\nabla\bfv)\big):
  \big(\nabla\bfv-\nabla\bfpsi_h\big)\,dx.
\end{align*}
As before this implies
\begin{align*}
  \norm{\bfF_{\mathcal T}(\cdot,\nabla\bfv) - \bfF_{\mathcal
      T}(\cdot,\nabla\tilde\bfv_h)}_2^2 &\leq c\,\norm{\bfF_{\mathcal
      T}(\cdot,\nabla\bfv) - \bfF_{\mathcal
      T}(\cdot,\nabla\bfpsi_h)}_2^2
  \\
  &\quad +c\,
  \int_\Omega\big(\bfAT(\cdot,\nabla\bfv)-\bfA(\cdot,\nabla\bfv)\big):
  \big(\nabla\bfv-\nabla\tilde{\bfv}_h\big)\,dx
  \\
  &\quad +c\,
  \int_\Omega\big(\bfA(\cdot,\nabla\bfv)-\bfAT(\cdot,\nabla\bfv)\big):
  \big(\nabla\bfv-\nabla\bfpsi_h\big)\,dx
  \\
  &=: (I) + (II) + (III).
\end{align*}
In particular, we get a best approximation result with the two
additional error terms $(II)$ and $(III)$. We begin in the following
with~$(II)$.

To estimate the difference between $\bfAT$ and $\bfA$ we need the
estimate
\begin{align*}
  \lefteqn{\bigabs{\bfA_{\mathcal T}(x,\bfQ)-\bfA(x,\bfQ)}} \quad &
  \\
  &\leq c\, \abs{p_{\mathcal{T}}(x) - p(x)}\, \abs{\ln(\kappa+
    \abs{\bfQ})} \Big( (\kappa + \abs{\bfQ})^{p_{\mathcal{T}}(x)-2} +
  (\kappa + \abs{\bfQ})^{p(x)-2}\Big) \abs{\bfQ}
  \\
  &\leq c\, h^\alpha \abs{\ln(\kappa+
    \abs{\bfQ})} \Big( (\kappa + \abs{\bfQ})^{p_{\mathcal{T}}(x)-2} +
  (\kappa + \abs{\bfQ})^{p(x)-2}\Big) \abs{\bfQ}
\end{align*}
for all $\bfQ \in \RNn$ using also that $p \in C^{0,\alpha}$. Hence,
we get
\begin{align*}
  (II)&:=\int_\Omega\big(\bfA_{\mathcal
    T}(\cdot,\nabla\bfv)-\bfA(\cdot,\nabla\bfv)\big):
  \big(\nabla\bfv-\nabla\tilde{\bfv}_h\big)\,dx
  \\
  &\leq c\, h^\alpha\!\!\int_\Omega \abs{\ln(\kappa\!+\!
    \abs{\nabla\bfv})} (\kappa \!+\!
  \abs{\nabla\bfv})^{p_{\mathcal{T}}(x)-2} \abs{\nabla \bfv}
  \abs{\nabla\bfv\!-\!\nabla\tilde{\bfv}_h}\,dx
  \\
  &\quad + c\, h^\alpha\!\!\int_\Omega \abs{\ln(\kappa\!+\!
    \abs{\nabla\bfv})} (\kappa \!+\!  \abs{\nabla\bfv})^{p(x)-2}
  \abs{\nabla \bfv} \abs{\nabla\bfv\!-\!\nabla\tilde{\bfv}_h}\,dx
  \\
  &=: (II)_1+(II)_2.
\end{align*}
We begin with the estimate for~$(II)_1$ on each~$K \in
\mathcal{T}$. Define the N-function
\begin{align*}
  \phi^K(t) := \int_0^t (\kappa + s)^{p_K-2}s \,ds.
\end{align*}
Using this definition we estimate
\begin{align*}
  (II)_1 &\leq c\, \sum_{K \in \mathcal{T}} \int_K h^\alpha
  \abs{\ln(\kappa\!+\!  \abs{\nabla\bfv})} (\phi^K)'( \abs{\nabla
    \bfv}) \abs{\nabla\bfv\!-\!\nabla\tilde{\bfv}_h}\,dx.
\end{align*}
Using Young's inequality with $\phi^K_{\abs{\nabla u}}$ on
$\abs{\nabla\bfv\!-\!\nabla\tilde{\bfv}_h}$ and its complementary
function on the rest, we get
\begin{align*}
  (II)_1 &\leq \delta \sum_{K \in \mathcal{T}} \int_K
  (\phi^K)_{\abs{\nabla
      \bfv}}(\abs{\nabla\bfv\!-\!\nabla\tilde{\bfv}_h})\,dx
  \\
  &+ c_\delta\, \sum_{K \in \mathcal{T}} \int_K \big(
  (\phi^K)_{\abs{\nabla \bfv}} \big)^* \Big( h^\alpha
  \abs{\ln(\kappa\!+\!  \abs{\nabla\bfv})} (\phi^K)'( \abs{\nabla
    \bfv}) \Big) \,dx.
\end{align*}
Now we use Lemma~\ref{lem:hammer} for the first line and
Lemma~\ref{lem:shiftedindex} and Lemma~\ref{lem:shifted2}
(with $\lambda=h^\alpha \leq 1$ using $h \leq 1$) for the second line
to find
\begin{align*}
  (II)_1 &\leq \delta c \norm{\bfF_{\mathcal T}(\cdot,\nabla\bfv) -
    \bfF_{\mathcal T}(\cdot,\nabla\tilde\bfv_h)}_2^2
  \\
  &+ c_\delta\, \sum_{K \in \mathcal{T}} \int_K
  (1+\abs{\ln(\kappa\!+\!  \abs{\nabla\bfv})})^{\max \set{2,p_K'}}
  \big( (\phi^K)_{\abs{\nabla \bfv}} \big)^* \Big( h^\alpha (\phi^K)'(
  \abs{\nabla \bfv}) \Big) \,dx
  \\
  &\leq \delta c \norm{\bfF_{\mathcal T}(\cdot,\nabla\bfv) -
    \bfF_{\mathcal T}(\cdot,\nabla\tilde\bfv_h)}_2^2
  \\
  &+ c_\delta\, \sum_{K \in \mathcal{T}} h^{2\alpha} \int_K
  (1+\abs{\ln(\kappa\!+\!  \abs{\nabla\bfv})})^{\max \set{2,p_K'}}
  (\phi^K)( \abs{\nabla \bfv}) \,dx.
\end{align*}
The term $(II)_2$ is estimate similarly taking into account the extra
factor $(\kappa + \abs{\nabla \bfv})^{p(x)-p_\mathcal{T}(x)}$. We get
\begin{align*}
  (II)_2 &\leq \delta c \norm{\bfF_{\mathcal T}(\cdot,\nabla\bfv) -
    \bfF_{\mathcal T}(\cdot,\nabla\tilde\bfv_h)}_2^2
  \\
  &\mspace{-30mu}+ c_\delta \sum_{K \in \mathcal{T}} h^{2\alpha}
  \!\!\int_K \big(1\!+\!\abs{\ln(\kappa\!+\!  \abs{\nabla\bfv})}
  (\kappa \!+\!  \abs{\nabla \bfv})^{p(x)-p_\mathcal{T}(x)}\big)^{\max
    \set{2,p_K'}} (\phi^K)( \abs{\nabla \bfv}) \,dx.
\end{align*}
Overall, this yields
\begin{align*}
  (II) &\leq \delta c\,\norm{\bfF_{\mathcal T}(\cdot,\nabla\bfv) -
    \bfF_{\mathcal T}(\cdot,\nabla\tilde\bfv_h)}_2^2 + c_\delta c_s
  h^{2\alpha} \int_\Omega \big(1 + \abs{\nabla \bfv}^{p(x)\,s}\big)
  \,dx,
\end{align*}
for some $s >1$. For $h$ small we can choose $s$
close to~$1$.

For $(III)$ the analogous estimate is
\begin{align*}
  (III) &\leq  c\,\norm{\bfF_{\mathcal T}(\cdot,\nabla\bfv) -
    \bfF_{\mathcal T}(\cdot,\nabla\bfpsi_h)}_2^2 +  c_s
 h^{2\alpha} \int_\Omega \big(1 + \abs{\nabla
    \bfv}^{p(x)\,s}\big) \,dx.
\end{align*}

We finally end up with the following result.
\begin{lemma}[best approximation]
  \label{lem:frozencea}
  Let $p\in C^{0,\alpha}(\overline{\Omega})$ for some $\alpha \in
  (0,1]$ with $p^{-}>1$. Let $\bfv \in (W^{1, \px}_0(\Omega))^N$
  and $\tilde{\bfv}_h \in \bfV_{h,0}$ be the solutions of
  \eqref{eq:weak} and \eqref{eq:weak_h'} respectively.  Then for some
  $s >1$ (close to~$1$ for $h$ small)
  \begin{align*}
    \norm{\bfF_{\mathcal T}(\cdot,\nabla\bfv) - \bfF_{\mathcal
        T}(\cdot,\nabla\tilde{\bfv}_h)}_2 &\le c \, \min _{\bfpsi_h
      \in \bfV_{h,0}} \norm{\bfF_{\mathcal T}(\cdot,\nabla \bfv) -
      \bfF_{\mathcal T}(\cdot,\nabla \bfpsi_h)}_2
    \\
    &+ c\,h^\alpha \bigg(\int_\Omega \big(1 + \abs{\nabla
      \bfv}^{p(x)\,s}\big) \,dx \bigg)^{\frac 12}.
  \end{align*}
\end{lemma}
We estimate the best approximation error by the projection error using
$\psi_h = \Pi_h \bfv$.
\begin{align*}
  \norm{\bfF_{\mathcal T}(\cdot,\nabla\bfv) - \bfF_{\mathcal T}(\cdot,\nabla\Pi_h \bfv)}_2 &\leq
  \norm{\bfF(\cdot,\nabla\bfv) - \bfF(\cdot,\nabla\Pi_h \bfv)}_2
  \\
  &\quad+ \norm{\bfF_{\mathcal T}(\cdot,\nabla\bfv) -
    \bfF(\cdot,\nabla \bfv)}_2
  \\
  &\quad + \norm{\bfF_{\mathcal T}(\cdot,\nabla\Pi_h \bfv) - \bfF
    (\cdot,\nabla\Pi_h \bfv)}_2
  \\
  &=: [I] + [II] + [III].
\end{align*}
We already know that $[I] \leq c\, h^\alpha$ by
Theorem~\ref{thm:error}. The estimate for $[II]$ and $[III]$ are
similar. Analogously to the estimate for~$\abs{\bfA_{\mathcal
    T}(x,\bfQ)-\bfA(x,\bfQ)}$ we have
\begin{align*}
  \lefteqn{\bigabs{\bfF_{\mathcal T}(x,\bfQ)-\bfF(x,\bfQ)}} \quad &
  \\
  &\leq c\, \abs{p_{\mathcal{T}}(x) - p(x)}\, \abs{\ln(\kappa+
    \abs{\bfQ})} \Big( (\kappa +
  \abs{\bfQ})^{\frac{p_{\mathcal{T}}(x)-2}{2}} + (\kappa +
  \abs{\bfQ})^{\frac{p(x)-2}{2}}\Big) \abs{\bfQ}.
\end{align*}
This implies as above
\begin{align*}
  [II] &\leq c\, h^\alpha \bigg(\int_\Omega (1+ \abs{\nabla
    \bfv})^{sp(x)}\,dx \bigg)^{\frac 12},
  \\
  [III] &\leq c\, h^\alpha \bigg(\int_\Omega (1+ \abs{\nabla \Pi_h
    \bfv})^{sp(x)}\,dx \bigg)^{\frac 12}.
\end{align*}
We can use the stability of~$\Pi_h$, see Lemma~\ref{thm:stabo} (for
the $a=0$ and the exponent~$sp(\cdot)$) to get
\begin{align*}
  [III] &\leq c\, h^\alpha \bigg(\int_\Omega (1+ \abs{\nabla
    \bfv})^{sp(x)}\,dx \bigg)^{\frac 12}.
\end{align*}
We summarize our calculations in the following theorem.
\begin{theorem}
\label{thm:frozen}
  Let $p\in C^{0,\alpha}(\overline{\Omega})$ with $p^{-}>1$. Let $\bfv
  \in (W^{1, \px}_0(\Omega))^N$ and $\tilde{\bfv}_h \in \bfV_{h,0}$ be
  solutions of \eqref{eq:weak} and \eqref{eq:weak_h'} respectively.
  Then
  \begin{align*}
    \bignorm{\bfF_{\mathcal{T}} (\cdot,\nabla\bfv) -
      \bfF_{\mathcal{T}} (\cdot,\nabla \tilde{\bfv}_h)}_2 &\leq c \,
    \min _{\bfpsi_h \in \bfV_{h,0}}
    \norm{\bfF_{\mathcal{T}}(\cdot,\nabla \bfv) -
      \bfF_{\mathcal{T}}(\cdot,\nabla \bfpsi_h)}_2 + c\, h^\alpha
    \\
    &\leq c\, \bignorm{\bfF_{\mathcal{T}} (\cdot,\nabla\bfv) -
      \bfF_{\mathcal{T}} (\cdot,\nabla \Pi_h \bfv)}_2+c\, h^\alpha
    \\
    &\leq c\, h^{\alpha},
  \end{align*}
  where $c$ also depends on $\norm{\bfF(\nabla \bfv)}_{1,2}$.
\end{theorem}

\appendix

\section{Orlicz spaces}
\label{sec:Orlicz spaces}
\noindent
The following definitions and results are standard in the theory of
Orlicz spaces and can for example be found in~\cite{KraR61,RaoR91}.
A continuous, convex function $\rho\,:\, [0,\infty) \to [0,\infty)$
with $\rho(0)=0$, and $\lim_{t \to \infty} \rho(t) = \infty$ is called
a {\em continuous, convex $\phi$-function}.  

We say that $\phi$
satisfies the $\Delta_2$--condition, if there exists $c > 0$ such that
for all $t \geq 0$ holds $\phi(2t) \leq c\, \phi(t)$. By
$\Delta_2(\phi)$ we denote the smallest such constant. Since $\phi(t)
\leq \phi(2t)$ the $\Delta_2$-condition is equivalent to $\phi(2t)
\sim \phi(t)$ uniformly in $t$. For a family $\phi_\lambda$ of
continuous, convex $\phi$-functions we define
$\Delta_2(\set{\phi_\lambda}) := \sup_\lambda \Delta_2(\phi_\lambda)$.
Note that if $\Delta_2(\phi) < \infty$ then $\phi(t) \sim \phi(c\,t)$
uniformly in $t\geq 0$ for any fixed $c>0$.
By $L^\phi$ and $W^{k,\phi}$, $k\in \setN_0$,  we denote the classical Orlicz and
Orlicz-Sobolev spaces, i.e.\ $f \in L^\phi$ iff $\int
\phi(\abs{f})\,dx < \infty$ and $f \in W^{k,\phi}$ iff $ \nabla^j f
\in L^\phi$, $0\le j\le k$.\\
A $\phi$-function $\rho$ is called a $N$-function iff it is strictly increasing and convex with
\begin{align*}
  \lim_{t\rightarrow0}\frac{\rho(t)}{t}=
  \lim_{t\rightarrow\infty}\frac{t}{\rho(t)}=0.
\end{align*}
By $\rho^*$ we denote the conjugate N-function of $\rho$, which is
given by $\rho^*(t) = \sup_{s \geq 0} (st - \rho(s))$. Then $\rho^{**}
= \rho$.
\begin{lemma}[Young's inequality]
  \label{lem:young}
  Let $\rho$ be an N-function. Then for all $s,t\geq 0$ we have
  \begin{align*}
    st \leq \rho(s)+\rho^*(t).
  \end{align*}
  If $\Delta_2(\rho,\rho^*)< \infty$, then additionally for all $\delta>0$
  \begin{align*}
    st &\leq \delta\,\rho(s)+c_\delta\,\rho^*(t)\text{ and }
    st \leq c_\delta\,\rho(s)+\delta\,\rho^*(t),
\\
\rho'(s)t &\leq \delta\,\rho(s)+c_\delta\,\rho(t)\text{ and }
\rho'(s)t \leq \delta\,\rho(t)+c_\delta\,\rho(s),
  \end{align*}
  where $c_\delta= c(\delta, \Delta_2(\set{\rho,\rho^*}))$.
\end{lemma}

\begin{definition}
  \label{ass:phipp}
  Let $\rho$ be an N-function. 
  We say that $\rho$ is {\em elliptic}, if
  $\rho$ is $C^1$ on $[0,\infty)$ and $C^2$ on $(0,\infty)$ and
  assume that 
  \begin{align}
    \label{eq:phipp}
    \rho'(t) &\sim t\,\rho''(t)
  \end{align}
  uniformly in $t > 0$. The constants hidden in $\sim$ are called the
  {\em characteristics of~$\rho$}.
\end{definition}
Note that~\eqref{eq:phipp} is stronger than
$\Delta_2(\rho,\rho^*)<\infty$. In fact, the $\Delta_2$-constants can
be estimated in terms of the characteristics of~$\rho$.

Associated to an elliptic $N$-function $\rho$ we define the tensors
\begin{align*}
  \bfA^\rho(\bfxi):=\frac{\rho'(\abs{\bfxi})}{\abs{\bfxi}}\bfxi,\quad
  \bfxi\in\mathbb R^{N\times n}\text{ and }
  \bfF^\rho(\bfxi):=\sqrt{\frac{\rho'(\abs{\bfxi})}{\abs{\bfxi}}}\,\bfxi,\quad
  \bfxi\in\mathbb R^{N\times n}.
\end{align*}

We define the {\em shifted} $N$-function $\rho_a$ for $a\geq 0$ by
\begin{align}
  \label{eq:def_shift}
  \rho_a(t) &:= \int_0^t \frac{\rho'(a+\tau)}{a+\tau} \tau\,d\tau.
\end{align}

The following auxiliary result can be found in~\cite{DieE08,DieK08,DieKapSch11,RuzDie07}
and~\cite[Lemma~11]{BelDieKre11}.
\begin{lemma}
  \label{lem:shift_sim}
  For all $a,b, t \geq 0$ we have
  \begin{align*}
    \rho_a(t) &\sim 
    \begin{cases}
      \rho''(a) t^2 &\qquad\text{if $t \lesssim a$}
      \\
      \rho(t) &\qquad\text{if $t \gtrsim a$,}
    \end{cases}
\\
(\rho_a)_b(t)&\sim\rho_{a+b}(t).
  \end{align*}
\end{lemma}
\begin{lemma}[{\cite[Lemma~2.3]{DieE08}}]
  \label{lem:hammer}
  We have
  \begin{align*}
    \begin{aligned}
      \big({\bfA^\rho}(\bfP) - {\bfA^\rho}(\bfQ)\big) \cdot
      \big(\bfP-\bfQ \big) &\sim \bigabs{ \bfF^\rho(\bfP) -
        \bfF^\rho(\bfQ)}^2
      \\
      &\sim \rho_{\abs{\bfP}}(\abs{\bfP - \bfQ})
      \\
      &\sim \rho''\big( \abs{\bfP} + \abs{\bfQ}
      \big)\abs{\bfP - \bfQ}^2
    \end{aligned}
  \end{align*}
  uniformly in $\bfP, \bfQ \in \setR^{n \times N}$.  Moreover,
  uniformly in $\bfQ \in \setR^{n \times N}$,
  \begin{align*}
    \bfA^\rho(\bfQ) \cdot \bfQ &\sim \abs{\bfF^\rho(\bfQ)}^2\sim
    \rho(\abs{\bfQ})\\
  \abs{{\bfA^\rho}(\bfP) - {\bfA^\rho}(\bfQ)}&\sim\big(\rho_{\abs{\bfP}}\big)'(\abs{\bfP - \bfQ}).
  \end{align*}
  The constants depend only on the characteristics of $\rho$.
\end{lemma} 
\begin{lemma}[Change of Shift]
  \label{lem:shift_ch}
  Let $\rho$ be an elliptic N-function. Then for each $\delta>0$ there
  exists $C_\delta \geq 1$ (only depending on~$\delta$ and the
  characteristics of~$\rho$) such that
  \begin{align*}
    \rho_{\abs{\bfa}}(t)&\leq C_\delta\, \rho_{\abs{\bfb}}(t)
    +\delta\, \rho_{\abs{\bfa}}(\abs{\bfa - \bfb}),
    \\
    (\rho_{\abs{\bfa}})^*(t)&\leq C_\delta\, (\rho_{\abs{\bfb}})^*(t)
    +\delta\, \rho_{\abs{\bfa}}(\abs{\bfa - \bfb}),
  \end{align*}
  for all $\bfa,\bfb\in\setR^d$ and $t\geq0$.
\end{lemma}
The case $\bfa=0$ or $\bfb=0$ implies the following corollary.
\begin{corollary}[Removal of Shift]
  \label{cor:shift_ch}
  Let $\rho$ be an elliptic N-function. Then for each $\delta>0$ there
  exists $C_\delta \geq 1$ (only depending on~$\delta$ and the
  characteristics of~$\rho$) such that
  \begin{align*}
    \rho_{\abs{\bfa}}(t)&\leq C_\delta\, \rho(t) +\delta\,
    \rho(\abs{\bfa}),
    \\
    \rho(t)&\leq C_\delta\, \rho_{\abs{\bfa}}(t) +\delta\,
    \rho(\abs{\bfa}),
  \end{align*}
  for all $\bfa\in\setR^d$ and $t\geq0$.
\end{corollary}
\begin{lemma}
  \label{lem:shifted2}
  Let $\rho$ be an elliptic N-function. Then $(\rho_a)^*(t) \sim
  (\rho^*)_{\rho'(a)}(t)$ uniformly in $a,t \geq 0$. Moreover, for all
  $\lambda \in [0,1]$ we have
  \begin{align*}
    \rho_a(\lambda a) &\sim \lambda^2 \rho(a) \sim
    (\rho_a)^*(\lambda \rho'(a)).
  \end{align*}
\end{lemma}
\begin{lemma}
  \label{lem:shiftedindex}
  Let $\rho(t) := \int_0^t (\kappa+s)^{q-2} s\,ds$ with $q \in
  (1,\infty)$ and $t\geq 0$. Then 
  \begin{align*}
    \rho_a(\lambda t) &\leq c\, \max\set{\lambda^q, \lambda^2}
    \rho(t),
    \\
    (\rho_a)^*(\lambda t) &\leq c\, \max\set{\lambda^{q'}, \lambda^2}
    \rho(t)
  \end{align*}
  uniformly in $a,\lambda \geq 0$.
\end{lemma}
\begin{remark}
  Let $p\in\PP(\Omega)$ with $p^{-}>1$ and $p^+<\infty$.  The
  results above extend to the function $\phi(x,t)=\int_0^t
  (\kappa+s)^{p(x)-2}s\,ds$ uniformly in $x\in\Omega$, where the
  constants only depend on $p^-$ and $p^+$.
\end{remark}

\newcommand{\etalchar}[1]{$^{#1}$}
\def\polhk#1{\setbox0=\hbox{#1}{\ooalign{\hidewidth
  \lower1.5ex\hbox{`}\hidewidth\crcr\unhbox0}}}
  \def\ocirc#1{\ifmmode\setbox0=\hbox{$#1$}\dimen0=\ht0 \advance\dimen0
  by1pt\rlap{\hbox to\wd0{\hss\raise\dimen0
  \hbox{\hskip.2em$\scriptscriptstyle\circ$}\hss}}#1\else {\accent"17 #1}\fi}
  \def\ocirc#1{\ifmmode\setbox0=\hbox{$#1$}\dimen0=\ht0 \advance\dimen0
  by1pt\rlap{\hbox to\wd0{\hss\raise\dimen0
  \hbox{\hskip.2em$\scriptscriptstyle\circ$}\hss}}#1\else {\accent"17 #1}\fi}
  \def\ocirc#1{\ifmmode\setbox0=\hbox{$#1$}\dimen0=\ht0 \advance\dimen0
  by1pt\rlap{\hbox to\wd0{\hss\raise\dimen0
  \hbox{\hskip.2em$\scriptscriptstyle\circ$}\hss}}#1\else {\accent"17 #1}\fi}
  \def\ocirc#1{\ifmmode\setbox0=\hbox{$#1$}\dimen0=\ht0 \advance\dimen0
  by1pt\rlap{\hbox to\wd0{\hss\raise\dimen0
  \hbox{\hskip.2em$\scriptscriptstyle\circ$}\hss}}#1\else {\accent"17 #1}\fi}
  \def\cprime{$'$}
\providecommand{\bysame}{\leavevmode\hbox to3em{\hrulefill}\thinspace}
\providecommand{\MR}{\relax\ifhmode\unskip\space\fi MR }
\providecommand{\MRhref}[2]{%
  \href{http://www.ams.org/mathscinet-getitem?mr=#1}{#2}
}
\providecommand{\href}[2]{#2}


\end{document}